\documentclass[12pt]{book}
\usepackage{amsthm,amsmath}
\usepackage[margin=1.0in]{geometry}
\usepackage{cleveref}

\usepackage{times}
\usepackage{newtxmath}

\newtheorem{problem}{Problem}
\newtheorem{conjecture}{Conjecture}
\newtheorem{theorem}{Theorem}

\theoremstyle{definition}
\newtheorem{definition}{Definition}
\newtheorem{reference}{Reference}
\newtheorem{assumption}{Assumption}
\newtheorem{setting}{Setting}
\newtheorem{remark}{Remark}

\begin{document}

\title{$\infty$-Categorical Generalized Langlands Correspondence II: Langlands Program Formalism}
\author{Xin Tong}
\date{}

\maketitle

\newpage
\subsection*{Abstract}
\noindent We extend the Langlands program in various subprograms with certain different generalizations: (1) Mixed-parity functorial perturbation of the usual Langlands program after Fargues-Scholze in all characteristics; (2) Robba-Frobenius sheafified functorial perturbation of the usual Langlands program after Fargues-Scholze and Kedlaya-Liu in all characteristics; (3) Arithmetic $D$-module theoretic functorial perturbation of the usual Langlands program after Fargues-Scholze and Abe-Kedlaya-Xu. In certain localized setting we construct Langlands correspondence from smooth representation side to Weil side through geometrized generalized Bernstein center. We also discuss certain generalization of the motivic local Langlands program after Scholze.

\newpage
\tableofcontents

\newpage
\begin{reference} $\blacksquare$ \cite{DHKM}, \cite{La} for Langlands program, \cite{Ta} for $p$-adic analysis, \cite{DI} for geometric Langlands program, \cite{DII} for geometric Langlands program, \cite{LLa} for geometric Langlands program, \cite{Fon} for $p$-adic analysis, \cite{BS} for the idea of the generalization herein, \cite{VLa} for geometric Langlands program, \cite{GL} for geometric local Langlands program, \cite{SchI} for perfectoid spaces, \cite{SchII} for modern $p$-adic analysis, \cite{SchIII} for analytic stacks, \cite{KLI} for $p$-adic analysis, \cite{KLII} for $p$-adic analysis, \cite{FS} for generalized local Langlands through deep geometrization, \cite{CSI} for condensed mathematics, \cite{CSII} for condensed analytic geometry, \cite{KI} for Drinfeld lemma for isocrystals, \cite{KXII} for Drinfeld lemma for isocrystals, \cite{AI} generalized Langlands for isocrystals in the $\mathrm{GL}$ situation, \cite{DK}, \cite{XZ}, \cite{Z}, \cite{Scho1}, \cite{Scho2}, \cite{V1}, \cite{A1}, \cite{A2}, \cite{A3}, \cite{A4}, \cite{G}, \cite{RS}.
\end{reference}

\newpage
\chapter{Generalized Langlands Program}

\section{Generalized Geometric Langlands Dualit\'e: The Motivation}

\noindent Let $X$ be an algebraic curve over a field. Let 
\begin{align}
G\longrightarrow X
\end{align}
be any algebraic reductive group over $X$. In the philosophy of \cite{La}, one can consider relating the following two seemingly hyper-unrelated $\infty$-categories. The first category is the corresponding $\infty$-category of $\ell$-adic perverse complexes over the moduli Artin stack of $G$-bundles over $X$. The second category is the corresponding derived $\infty$-category of coherent sheaves over the moduli stack of $G^\text{dualit\'e}$-local systems over $X$ in some coefficient with characteristic which is not the same as the base field (such as in real algebraic geometry, complex algebraic geometry, functional field situation). Dual Langlands group here represents the group flipping the reductive datum. We denote them as:
\begin{align}
\mathrm{DCate}_X(\mathrm{Bundle}_G,A)
\end{align}
and
\begin{align}
\mathrm{DCate}(\mathrm{Stack}_{X,\mathrm{local}},\mathcal{O}^{G^\text{dualit\'e}}).
\end{align}
These are very usual geometric objects in the corresponding classical algebraic geometry. Say you take any curve then you can consider those categories as above. Langlands Program is a program looking at certain correspondence between the corresponding two $\infty$-categories above. In certain scenario we expect that they are equivalent in certain well-esablished context. Suppose we now consider the profinite \'etale fundamental group of $X$, $\Pi_X$, and we consider a two-fold covering $\Pi_{X,2}$ of this group in the category of all the profinite groups, then we consider the corresponding moduli Artin stack of the morphism:
\begin{align}
\Pi_{X,2} \rightarrow G^\text{dualit\'e}.
\end{align}
We then use the corresponding notation:
\begin{align}
\mathrm{DCate}(\mathrm{Stack}_{X,\mathrm{local}},\mathcal{O}^{G^\text{dualit\'e}}).
\end{align}
to denote the corresponding coherent sheaves over this stack.

\begin{problem}
What should we modify for the $\mathrm{Bundle}_G$ side to correspond to the modification on considering covering of profinite \'etale fundamental groups? To be more precise, fixing the original Langlands functoriality and then starting with a representation:
\begin{align}
\Pi_{X,2} \rightarrow G^\text{dualit\'e},
\end{align}
what kind of geometric object on the other side will correspond to this representation? 
\end{problem}
The author currently does \textit{not} have an answer to this globalized generalized Langlands conjecture. You might wonder what if we do the infinitesimal localization to some point over the curves $X$, say $x$. Then following \cite{GL} we do actually have the local duality, which ends up with the following question. $\mathrm{Bundle}^\mathrm{local}_G$ will be the Artin algebraic stack sending any scheme to the corresponding groupoid of the corresponding vector bundles (with $G$-bundle structure) over $\mathrm{Spec}\mathcal{O}_{X,x}$, i.e. the corresponding algebraic localization of $X$ around $x$. Taking the formal completion we have the corresponding Artin formal stack $\mathrm{Bundle}^\mathrm{local,formal}_G$.

\begin{problem}\mbox{}\\
Suppose our curve now is over $\mathbb{F}_q$. What should we modify for the $\mathrm{Bundle}^\mathrm{local}_G$ (or $\mathrm{Bundle}^\mathrm{local,formal}_G$) side to correspond to the modification on considering covering of profinite \'etale fundamental groups? To be more precise, fixing the original Langlands functoriality and then starting with a representation:
\begin{align}
\Pi_{X,\mathcal{O}_{X,x},2} \rightarrow G^\text{dualit\'e},
\end{align}
what kind of geometric object on the other side will correspond to this representation? 
\end{problem}

The author still cannot answer this question, but the infinitesimal localization in the spirit of modern relative $p$-adic Hodge theory will give some answer, see the following chapters.

\begin{remark}
Over $\mathbb{F}_1$, the picture is completely unknown. Therefore the corresponding version of the problem above in the philosophy of Langlands is completely unknown as well.
\end{remark}

\indent After Abe, Kedlaya and Kedlaya-Xu \cite{AI}, \cite{KI} and \cite{KXII} one can actually consider the following picture. $\mathrm{Bundle}^\mathrm{local}_G$ will be the Artin algebraic stack sending any scheme to the corresponding groupoid of the corresponding vector bundles (with $G$-bundle structure) over $\mathrm{Spec}\mathcal{O}_{X,x}$, i.e. the corresponding algebraic localization of $X$ around $x$. Taking the formal completion we have the corresponding Artin formal stack $\mathrm{Bundle}^\mathrm{local,formal}_G$. We then consider the corresponding derived category of all the arithmeic $D$-modules in this setting which is actually again compatible with \cite{FS} and \cite{GL}. Although essentially speaking \cite{FS} and \cite{GL} did use different stacks in the corresponding different categories. This category is actually simpler in this local setting, since the corresponding vector bundles are basically over a formal disk which is actually presenting the corresponding category of arithmetic $D$-modules in a very transparent way.

\begin{problem}\mbox{}\\
Suppose our curve now is over $\mathbb{F}_q$. What should we modify for the $\mathrm{Bundle}^\mathrm{local}_G$ (or $\mathrm{Bundle}^\mathrm{local,formal}_G$) side to correspond to the modification on considering covering of profinite \'etale fundamental groups? To be more precise, fixing the original Langlands functoriality and then starting with a representation:
\begin{align}
\Pi^\mathrm{Tannakian}_{X,\mathcal{O}_{X,x},2} \rightarrow G^\text{dualit\'e},
\end{align}
what kind of geometric object on the other side will correspond to this representation in the corresponding category of isocrytals? Here the corresponding Tannakian groups are the Tannakian groups for the overconvergent isocrystals in the local setting as in \cite{AI}, \cite{KI}, \cite{KXII}, \cite{DK}.
\end{problem}

\newpage
\section{Main Theorems}

Following essentially Langlands philosophy in \cite{La}, we extend the Langlands program in various subprograms with certain different generalizations: (1) Mixed-parity functorial perturbation of the usual Langlands program after Fargues-Scholze in all characteristics; (2) Robba-Frobenius sheafified functorial perturbation of the usual Langlands program after Fargues-Scholze and Kedlaya-Liu in all characteristics; (3) Arithmetic $D$-module theoretic functorial perturbation of the usual Langlands program after Fargues-Scholze and Abe-Kedlaya-Xu. In certain localized setting we construct Langlands correspondence from smooth representation side to Weil side through geometrized generalized Bernstein center. We first have the following results on the cohomologization formalism of one side of the correspondence, namely the smooth representation side.

\begin{theorem}
The $v$-prestack as below $\mathrm{Bundle}_{G,2}$ is actually $v$-stack, which is furthermore small as in \cite{FS}, and is furthermore Artin stack in the category of all the $v$-stacks.
\end{theorem}

\begin{theorem}
For any finite extension $K$ over $\mathbb{Q}_p$ or $\mathbb{F}_q((z))$ with $\ell\neq p$, $p>2$ we have well-defined Schur-irreducible sheaves in the corresponding derived $\infty$-category $D(\mathrm{Bundle}_{G,2})_{\text{lisse},\blacksquare,A}$ which are containing all the corresponding Schur-irreducible sheaves in the corresponding derived $\infty$-category $D(\mathrm{Bundle}_{G})_{\text{lisse},\blacksquare,A}$ in the usual setting in \cite{FS} in certain functorial manner, while the latters are related directly to the corresponding smooth representations of the corresponding reductive group $G(K)$ by considering the Bernstein center.
\end{theorem}

\begin{theorem}
For any finite extension $K$ over $\mathbb{Q}_p$ or $\mathbb{F}_q((z))$ with $\ell\neq p$, $p>2$ we have well-defined Schur-irreducible sheaves in the corresponding derived $\infty$-category $D(\mathrm{Bundle}_{G,2})_{\text{KL},\mathrm{lisse},\blacksquare,\widetilde{C}_\blacksquare}$ which are containing all the corresponding Schur-irreducible sheaves in the corresponding derived $\infty$-category $D(\mathrm{Bundle}_{G})_{\text{KL},\mathrm{lisse},\blacksquare,\widetilde{C}_\blacksquare}$ in the usual setting in \cite{FS} in certain functorial manner, while the latters are related directly to the corresponding smooth representations of the corresponding reductive group $G(K)$ by considering the Bernstein center.
\end{theorem}

\indent Then following \cite{VLa}, \cite{FS}, \cite{AI}, \cite{KI}, \cite{KXII} we have the following generalization. Now we consider the corresponding interaction between the two sides, where this interaction literally happens at the geometric Bernstein center.

\begin{theorem}
For any finite extension $K$ over $\mathbb{Q}_p$ or $\mathbb{F}_q((z))$ with $\ell\neq p$, $p>2$ we have well-defined Schur-irreducible sheaves in the corresponding derived $\infty$-category $D(\mathrm{Bundle}_{G,2})_{\text{lisse},\blacksquare,A}$ which are containing all the corresponding Schur-irreducible sheaves in the corresponding derived $\infty$-category $D(\mathrm{Bundle}_{G})_{\text{lisse},\blacksquare,A}$ in the usual setting in \cite{FS} in certain functorial manner, while the latters are related directly to the corresponding smooth representations of the corresponding reductive group $G(K)$ by considering the Bernstein center. Then for any such Schur-irreducible complex there exist some $W_{K,2}$-parameterization into Langlands dual group in well-defined functorial way, compatible with the usual Fargues-Scholze correspondence. Here we make the same requirement on $A$ as in \cite{FS} which is chosen to be an algebraic closure of the $\ell$-adic field containing the square root of the residual field cardinality.
\end{theorem}

\begin{theorem}
For any algebraic curve $X$ over finite field, and we fix a point $x\in X$. we have well-defined Schur-irreducible sheaves in the corresponding derived $\infty$-category $D(\mathrm{Bundle}_{G})_{\mathrm{arithmetic-D}}$. Then for any such Schur-irreducible complex there exist some $W^\mathrm{Tannakian}_{F_x}$-parameterization into Langlands dual group in well-defined functorial way, compatible with the usual Abe's correspondence for $\mathrm{GL}$. Here the corresponding Tannakian groups are the Tannakian groups for the overconvergent isocrystals in the local setting as in \cite{AI}, \cite{KI}, \cite{KXII}, \cite{DK}.
\end{theorem}

\newpage
\chapter{Stacky Approach to the Generalized Langlands Program}

\section{Generalization through Geometrization}

\noindent In this section we prove the following theorems:

\begin{theorem}
The $v$-prestack as below $\mathrm{Bundle}_{G,2}$ is actually $v$-stack, which is furthermore small as in \cite{FS}, and is furthermore Artin stack in the category of all the $v$-stacks.
\end{theorem}

\begin{theorem}
For any finite extension $K$ over $\mathbb{Q}_p$ or $\mathbb{F}_q((z))$ with $\ell\neq p$, $p>2$ we have well-defined Schur-irreducible sheaves in the corresponding derived $\infty$-category $D(\mathrm{Bundle}_{G,2})_{\text{lisse},\blacksquare,A}$ which are containing all the correspnding Schur-irreducible sheaves in the corresponding derived $\infty$-category $D(\mathrm{Bundle}_{G})_{\text{lisse},\blacksquare,A}$ in the usual setting in \cite{FS} in certain functorial manner, while the latters are related directly to the corresponding smooth representations of the corresponding reductive group $G(K)$ by considering the Bernstein center.
\end{theorem}

\begin{theorem}
For any finite extension $K$ over $\mathbb{Q}_p$ or $\mathbb{F}_q((z))$ with $\ell\neq p$, $p>2$ we have well-defined Schur-irreducible sheaves in the corresponding derived $\infty$-category $D(\mathrm{Bundle}_{G,2})_{\text{KL},\mathrm{lisse},\blacksquare,\widetilde{C}_\blacksquare}$ which are containing all the corresponding Schur-irreducible sheaves in the corresponding derived $\infty$-category $D(\mathrm{Bundle}_{G})_{\text{KL},\mathrm{lisse},\blacksquare,\widetilde{C}_\blacksquare}$ in the usual setting in \cite{FS} in certain functorial manner, while the latters are related directly to the corresponding smooth representations of the corresponding reductive group $G(K)$ by considering the Bernstein center.
\end{theorem}

\noindent We now answer the problem from the perspectives of certain infinitesimal localization from the main development of \cite{VLa}, \cite{FS}. The corresponding consideration is basically considering the corresponding a slightly different version of the stacks in the category of certain analytic stacks. This certainly causes further foundational problems, since the topologization and functional analytification in our current setting cannot be easily translated to the corresponding algebraic setting. This is actually quite clear say in the corresponding formal scheme setting but in that setting the corresponding difficulty is not so heavy by directly mimicking the corresponding algebraic theory. In the analytic geometry the problem can be very tricky for instance the corresponding analytification cannot be very easily achieved from the algebraic stacky consideration through the obvious analytification. Another thing is that by considering further analytification consideration one can actually have more useful structures such as more Galois action, \'etale cohomological consideration and so on. Let us first consider the following assumption:

\begin{assumption}
Let $K$ be a finite extension of $\mathbb{Q}_p$ or $\mathbb{F}_q((z))$, where we assume that we have another prime number $\ell \neq p$ and $p>2$. We consider any reductive group $G/K$ over $K$ as in the corresponding consideration in \cite{FS}, this field $K$ is fixed whenever we talk about the corresponding generalized Langlands conjecture in our setting.
\end{assumption}

\begin{conjecture}
We first use the notation $G^\text{duali\'e}$ to denote the corresponding reductive dual group by flipping the corresponding reductive datum while we consider the first right cross product action from the Weil group $W_K$, where we use the corresponding notation $G^\text{Lan}$ to denote this:
\begin{align}
G^\text{duali\'e} \overset{\mathrm{right}}{\times} W_K.
\end{align}
For any condensed morphism from the corresponding group $W_{K,2}$ to this larger Langlands dual group where we use this notation to denote the two-fold covering group of $W_{K}$ for instance by either pullback along the map $G_{K,2}\rightarrow G_K$ or by considering the corresponding closure of the group $W_K$ in $G_{K,2}$, we conjecture there is certain geometric objects on the other side of the Langlands correspondence which will functorially extend the usual objects in \cite{FS}, i.e. the Schur-irreducible sheaves over the moduli $v$-stack of the corresponding vector bundles over Fargues-Fontaine curves.
\end{conjecture}

\indent This conjecture is literally saying that we can extend the picture in \cite{FS} in certain functorial way, which is essentially following the original Langlands functoriality conjectures. 

\begin{definition}
Now introduce the corresponding main objects in our setting. First we consider the corresponding two fold covering of the usual Fargues-Fontain curves in the following sense. We consider the category of all the perfectoid spaces over 
\begin{align}
\mathrm{Spd}\overline{\mathbb{F}}_p\otimes_{\mathbb{F}_p}\mathbb{Q}_p(\mu_{p^\infty})^{\wedge,\flat}.
\end{align}
This amounts to saying that all the small $v$-stacks in the following over this site will be taking the form of:
\begin{align}
\widetilde{X}\times_{\mathrm{Spd}\overline{\mathbb{F}}_p}\mathrm{Spd}\overline{\mathbb{F}}_p((k^{1/p^\infty}))
\end{align}
which is some \textit{smooth} base change from some $\widetilde{X}$. Here the functor $\mathrm{Spd}$ parametrizes the corresponding untilts of the corresponding perfectoid spaces in the equal characteristics. Let us recall that from \cite{SchI}, \cite{KLI}, \cite{KLII} this means that we have certain map from $Witt_{K}(X^\flat) \rightarrow X^\sharp$ which locally presents a principal kernel where the kernel is Cartier divisor generated by certain elememt $t$, this $t$ can be used to take further filtered completion of this generalized Witt vector ring which further produces the corresponding de Rham rings. For instance in the situation where $K$ is of mixed characteristic, we can take the any Tate perfectoid algebra $R$, over some $\mathbb{F}_p((\overline{z}))$, and then we have $t$ can be chosen to be $\log([\overline{z}+1])$, where $\varphi$ acts by some uniformizer and the Galois group of $K$ acts via $\mathcal{O}_K$ through cyclotomic character. We can then join the square root of $t$ in both these two stages, with the corresponding same notation $t^{1/2}$, which produces further extension of the de Rham rings and the corresponding Robba rings with each untilts. Then we take a finite extension of the ground field to allow the action $\varphi$ on this square root for $K$. Then we have the corresponding two-fold covering FF-stack $\mathrm{Stack}_{\mathrm{FF},2}$ with morphism:
\begin{align}
\mathrm{Stack}_{\mathrm{FF},2}\rightarrow \mathrm{Stack}_{\mathrm{FF}}.
\end{align}
Then we consider the moduli stack of vector bundles over these FF-curves. We use the notation $\mathrm{Bundle}_{G,2}$ to denote the $v$-prestack over the category (endowed with $v$-topology) of all the perfectoid spaces in the above: $\mathrm{PEF}_\mathrm{\mathrm{Spd}\overline{\mathbb{F}}_p}$ parametrizing the groupoid of the corresponding $G$-bundles over the $\mathrm{Stack}_{\mathrm{FF},2}$.
\end{definition}

\indent We then have the following key results in our setting closely after \cite[Chapter III, Chapter IV, Chapter V]{FS}:

\begin{theorem}
The $v$-prestack as above $\mathrm{Bundle}_{G,2}$ is actually $v$-stack, which is furthermore small as in \cite{FS}, and is furthermore Artin stack in the category of all the $v$-stacks.
\end{theorem}

\begin{proof}
We first mention that the corresponding stack here represents the corresponding morphism:
\begin{align}
\mathrm{Bundle}_{G,2} \rightarrow \mathrm{Bundle}_{G}
\end{align}
over the morphism in between the corresponding morphism of the stacks:
\begin{align}
\mathrm{Stack}_{\mathrm{FF},2}\rightarrow \mathrm{Stack}_{\mathrm{FF}}.
\end{align}
This consideration directly elaborates a corresponding commutative diagram which realizes the corresponding prestack $\mathrm{Bundle}_{G,2}$ as a corresponding fiber product of the corresponding stack $\mathrm{Bundle}_{G}$ with the stack $\mathrm{Stack}_{\mathrm{FF},2}$ directly. This will imply directly the corresponding smallness of our stack, which certainly can be proved by using the corresponding argument in \cite[Proposition 1.3 in Chapter III]{FS}, since that argument in itself can be formalized in our setting. To prove that our $v$-prestack is a stack one can proceed again in these two way fashion. One way is to consider the corresponding fiber product as above for the stack $\mathrm{Bundle}_{G,2}$ while the other way is to consider the corresponding argument as in the following following \cite{FS}. Namely we look at the derived higher order direct section functor where all the complexes over the extended Robba rings, while we can actually regard those sheaves as the corresponding sheaves over the usual Kedlaya-Liu Robba sheaves in the perfect setting, which will produce directly $v$-descent for the complexes of sheaves. This proves directly that the $v$-presheaf of moduli of $G$-bundles is in our setting corresponding $v$-stack as in \cite{FS}. For the property of being Artin, we consider the following two way approaches. First observe the following. The Kottwitz set $B(G)$ actually makes sense in our setting, which remains the same. One should be able to use this directly to construct the corresponding smooth chart in the definition of the corresponding Artin $v$-stack in the corresponding mixed-parity setting. However we can consider the usual chart $\mathrm{Stack}_b\rightarrow \mathrm{Bundle}_{G}$ to take the correpsonding fiber-product along the corresponding morphism:
\begin{align}
\mathrm{Bundle}_{G,2} \rightarrow \mathrm{Bundle}_{G}.
\end{align}
This produces certain chart in our setting:
\begin{align}
\mathrm{Stack}_{b,2}\rightarrow \mathrm{Bundle}_{G,2}.
\end{align}
Here $\mathrm{Bundle}_G$ is the key stack in \cite{FS} over the $v$-site of all the Tate perfectoid over
\begin{align}
\mathrm{Spd}\overline{\mathbb{F}}_p\otimes_{\mathbb{F}_p}\mathbb{Q}_p(\mu_{p^\infty})^{\wedge,\flat},
\end{align}
where equivalently one can take the moduli in \cite{FS} over $\overline{\mathbb{F}}_p$ and take the fiber product of this moduli over $\overline{\mathbb{F}}_p$ with $\mathrm{Spd}\overline{\mathbb{F}}_p((k^{1/p^\infty}))$, together with the smooth charts.
\end{proof}

\begin{remark}
Let us elaborate slightly more on the subtle point here. The corresponding mechanism allows us to consider the corresponding stability of $G$-bundles in our generalized setting, but in that way the whole mechanism looks a bit tedious and tricky if you look at the corresponding discussion in the usual situation and then to directly mimick. Therefore our strategy (in certain well-defined situation) is to consider the corresponding pull back of the semistable, and more sublocus inside our stack $\mathrm{Bundle}_{G,2}$. For instance we have the following definition.
\end{remark}

\begin{definition}
We define the mixed-parity version of the usual de Rham Grassmannian in our situation. Again over the category $\mathrm{PEF}_{\mathrm{\mathrm{Spd}}\overline{\mathbb{F}}_p((k^{1/p^\infty}))}$. Over this category each until will define certain Cartier divisor of rank 1 in the flavor of perfectoid spaces after Kedlaya-Liu and Scholze. The corresponding generalized Fontaine-Wintenberger correspondence procuces the map:
\begin{align}
Witt_K(X^\flat) \rightarrow X^\sharp
\end{align}
which further gives rise to the definition of the distinguished element $t$. One can call this the corresponding Cartier divisorial distinguished deformation in the generalized Fontaine-Wintenberger fashion. Then take the corresponding filtered completion with respect to $t$ we have the corresponding definition of the de Rham sheaves in our setting after we consider largen this sheaves by adding the square root of $t$. Then we can define as in \cite{FS} the corresponding $B_\mathrm{dR,2}$-Grassmannian $v$-stack, which can be directly defined as the corresponding pull-back along the corresponding morphism of small Artin $v$-stacks:
\begin{align}
\mathrm{Bundle}_{G,2} \rightarrow \mathrm{Bundle}_{G}.
\end{align}
Or one can follow the definition in \cite{FS} to give the definition by using the corresponding $B_\mathrm{dR,2}$-lattices, and the corresponding $B_\mathrm{dR,2}$-loop groups. \footnote{Be careful that here we are using the ring $\mathbb{C}_p[[k^{1/2}]][k^{-1/2}]\otimes *$, where $*$ is some finite extension of $\mathbb{Q}_p$ containing the corresponding square root of $\varphi(k^{1/2})$ and realizing the corresponding two fold covering of the local Galois group $G_{K,2}$. Whenever we use this notation we will assume the existence of $*$ as in \cite{BS}.}
\end{definition}

\begin{remark}
Be careful that here we are using the ring $\mathbb{C}_p[[k^{1/2}]][k^{-1/2}]\otimes *$, where $*$ is some finite extension of $\mathbb{Q}_p$ containing the corresponding square root of $\varphi(k^{1/2})$ and realizing the corresponding two fold covering of the local Galois group $G_{K,2}$. Whenever we consider joining the corresponding element $t^{1/2}$ to the corresponding period sheaves such as the de Rham or the Robba we will assume the existence of $*$ as in \cite{BS}. In more general setting and in the corresponding equal-characteristic setting we do have the corresponding same requirement namely not over this geometric point.

\end{remark}

\noindent Then we can consider the key derived $\infty$-categories in our setting following the corresponding \cite{FS}. Fargues-Scholze's development is heavily focus on the construction of these categories which eventually emerges in the corresponding Bernstein center construction after \cite{VLa} (while Lafforgue's result is extremely formalized for any $\mathbb{Z}_\ell$-category).

\begin{definition}
For any $\mathbb{Z}_\ell$-algebra $A$, after \cite{FS} we define directly (under our current context and the verifications) the following derived $\infty$-categories:
\begin{align}
&D(\mathrm{Bundle}_{G,2})_{\text{etale},A},\\
&D(\mathrm{Bundle}_{G,2})_{\text{lisse},\blacksquare,A},\\
&D(\mathrm{Bundle}_{G,2})_{\blacksquare,A}.
\end{align}
\end{definition}

\begin{theorem}
These categories:
\begin{align}
&D(\mathrm{Bundle}_{G,2})_{\text{etale},A},\\
&D(\mathrm{Bundle}_{G,2})_{\text{lisse},\blacksquare,A},\\
&D(\mathrm{Bundle}_{G,2})_{\blacksquare,A}
\end{align}
are well-defined.
\end{theorem}

\begin{proof}
We apply the corresponding heavy mechinery in \cite[Chapter V, Chapter VII]{FS} to the stack $\mathrm{Bundle}_{G,2}$.
\end{proof}

Now we can state our main theorem in the corresponding context as in the following:

\begin{theorem}
For any finite extension $K$ over $\mathbb{Q}_p$ or $\mathbb{F}_q((z))$ with $\ell\neq p$, $p>2$ we have well-defined Schur-irreducible sheaves in the corresponding derived $\infty$-category $D(\mathrm{Bundle}_{G,2})_{\text{lisse},\blacksquare,A}$ which are containing all the correspnding Schur-irreducible sheaves in the corresponding derived $\infty$-category $D(\mathrm{Bundle}_{G})_{\text{lisse},\blacksquare,A}$ in the usual setting in \cite{FS} in certain functorial manner, while the latters are related directly to the corresponding smooth representations of the corresponding reductive group $G(K)$ by considering the Bernstein center.
\end{theorem}

\begin{proof}
Putting all we have proved in this section we have the theorem follows.
\end{proof}

\noindent The corresponding $v$-sheaves consideration can be actually generalized to the corresponding Robba and Frobenius sheaves after Kedlay-Liu in \cite{KLII} and \cite{KLI}. Here we remind the reader that the one can have many ways to do the corresponding generalization. There is one way to consider the corresponding directly Grothendieck-categoricalization of the corresponding pseudo-coherent sheaves, which produces the corresponding Grothendieck categories which can be enhanced to be the corresponding derived $\infty$-category of all the corresponding \textit{quasi-coherent sheaves with Frobenius structure over Kedlaya-Liu Robba sheaf $\widetilde{C}_\blacksquare$}. However the other approach is to use the corresponding solid \textit{quasi-coherent sheaves with Frobenius structure over Kedlaya-Liu Robba sheaf $\widetilde{C}_\blacksquare$} after Clausen-Scholze, which is what we adopt.

\begin{definition}
After \cite{FS} we define directly (under our current context and the verifications) the following derived $\infty$-categories of all the solid quasi-coherent sheaves with Frobenius structure over Kedlaya-Liu Robba sheaf $\widetilde{C}_\blacksquare$:
\begin{align}
D(\mathrm{Bundle}_{G,2})_{\text{KL},\blacksquare,\widetilde{C}_\blacksquare},\\
\end{align}
with the well-defined sub category of all the lisse complexes objects generated from the locally free objects which are finitely generated (which are automatically condensed solid), which we denote it as:
\begin{align}
D(\mathrm{Bundle}_{G,2})_{\text{KL},\mathrm{lisse},\blacksquare,\widetilde{C}_\blacksquare},\\
\end{align}
\end{definition}

\begin{theorem}
These categories:
\begin{align}
&D(\mathrm{Bundle}_{G,2})_{\text{KL},\blacksquare,\widetilde{C}_\blacksquare},\\
&D(\mathrm{Bundle}_{G,2})_{\text{KL},\mathrm{lisse},\blacksquare,\widetilde{C}_\blacksquare}
\end{align}
are well-defined.
\end{theorem}

\begin{proof}
We apply the corresponding heavy mechinery in \cite[Chapter V, Chapter VII]{FS} to the stack $\mathrm{Bundle}_{G,2}$. Also we rely on \cite{SchIII}, \cite{CSI}, \cite{CSII}.
\end{proof}

Now we can state our main theorem in the corresponding context as in the following:

\begin{theorem}
For any finite extension $K$ over $\mathbb{Q}_p$ or $\mathbb{F}_q((z))$ with $\ell\neq p$, $p>2$ we have well-defined Schur-irreducible sheaves in the corresponding derived $\infty$-category $D(\mathrm{Bundle}_{G,2})_{\text{KL},\mathrm{lisse},\blacksquare,\widetilde{C}_\blacksquare}$ which are containing all the corresponding Schur-irreducible sheaves in the corresponding derived $\infty$-category $D(\mathrm{Bundle}_{G})_{\text{KL},\mathrm{lisse},\blacksquare,\widetilde{C}_\blacksquare}$ in the usual setting in \cite{FS} in certain functorial manner, while the latters are related directly to the corresponding smooth representations of the corresponding reductive group $G(K)$ by considering the Bernstein center.
\end{theorem}

\begin{proof}
Putting all we have proved in this section we have the theorem follows.
\end{proof}

\begin{remark}
This theorem is actually a double and second-further generalization of the corresponding consideration in \cite{FS}, namely the category $D(\mathrm{Bundle}_{G})_{\text{KL},\mathrm{lisse},\blacksquare,\widetilde{C}_\blacksquare}$ goes along a direct $p$-adic cohomologicalzation of the $\ell$-adic cohomologization perverse enough from \cite{FS}. The generalized category $D(\mathrm{Bundle}_{G,2})_{\text{KL},\mathrm{lisse},\blacksquare,\widetilde{C}_\blacksquare}$ carries the corresponding $W_{K,2}$-equivariance information which is very interesting to study.
\end{remark}

\begin{conjecture}
We have well-defined formalism in this context for the 6-functors in condensed quasi-coherent setting, namely if we have the preservance of the \textit{propert\'e lisse} in our current setting under the pull-back and the push-forward.
\end{conjecture}

\newpage
\section{Hecke Eigensheaves Formalization}

\noindent In this section we prove the following theorems:

\begin{theorem}
For the corresponding setting (1) and (2) we have the corresponding Hecke stacks $\mathrm{Stack}_{\mathrm{Hecke},G,2,I}$ for any finite set $I$. This Hecke stack can be defined by taking the corresponding pull-back along the corresponding morphism from $\mathrm{Bundle}_{G,2}$ to the corresponding stack $\mathrm{Bundle}_{G}$. We have then the corresponding morphisms in the following:
\begin{align}
f_A: \mathrm{Stack}_{\mathrm{Hecke},G,2,I} \rightarrow \mathrm{Bundle}_{G,2}
\end{align}
with
\begin{align}
f_B: \mathrm{Stack}_{\mathrm{Hecke},G,2,I} \rightarrow \mathrm{Bundle}_{G,2}\times \Pi_I \mathrm{Stack}_{\mathrm{Cartier},W_{K,2}}
\end{align}
with
\begin{align}
f_C: \mathrm{Stack}_{\mathrm{Hecke},G,2,I} \rightarrow \mathrm{Bundle}_{G,2}\times \Pi_I \mathrm{Stack}_{\mathrm{Cartier},W_{K,2}}\rightarrow \mathrm{Bundle}_{G,2}\times \mathrm{Stack}_{\mathrm{classify},\Pi_I W_{K,2}}.
\end{align}
The corresponding Hecke operation in this setting is well-defined. The corresponding image lies in the corresponding $W_{K,2}$-invariant sub-derived $\infty$-category:
\begin{align}
D(\mathrm{Bundle}_{G,2})_{\text{lisse},\blacksquare,A}.
\end{align}
For any representation $R$ of the Langlands group $G^\mathrm{Lan}$ in the coefficient $A$ we define the corresponding Hecke operator $\mathrm{Hecke}(\square)$ in mixed-parity setting by using the categoricalized Fourier-transformation fashion functor by first taking the corresponding pushforward along $f_A$ of any complex then take the condensed tensor product with the sheaf corresponding to $R$ over the Hecke stack, then take the corresponding push-forward along:
\begin{align}
f_B: D(\mathrm{Stack}_{\mathrm{Hecke},G,2,I})_{\text{lisse},\blacksquare,A} \rightarrow D(\mathrm{Bundle}_{G,2}\times \Pi_I \mathrm{Stack}_{\mathrm{Cartier},W_{K,2}})_{\text{lisse},\blacksquare,A}.
\end{align}
Namely we use the notation $\mathrm{Hecke}(\square)$ for:
\begin{align}
\mathrm{pushforward}_{f_B}(\mathrm{pullback}_{f_A}\square\otimes_{\blacksquare,\mathrm{complete}}\mathcal{F}_R).
\end{align}
For the corresponding setting (4) we have the corresponding Hecke stacks $\mathrm{Stack}_{\mathrm{Hecke},G,I}$ for any finite set $I$.We have then the corresponding morphisms in the following:
\begin{align}
f_A: \mathrm{Stack}_{\mathrm{Hecke},G,I} \rightarrow \mathrm{Bundle}_{G}
\end{align}
with
\begin{align}
f_B: \mathrm{Stack}_{\mathrm{Hecke},G,I} \rightarrow \mathrm{Bundle}_{G}\times \Pi_I \mathrm{Stack}_{\mathrm{Cartier},W^\mathrm{Tannakian}_{F_x}}
\end{align}
with
\begin{align}
f_C: \mathrm{Stack}_{\mathrm{Hecke},G,I} \rightarrow \mathrm{Bundle}_{G}\times \Pi_I \mathrm{Stack}_{\mathrm{Cartier},W^\mathrm{Tannakian}_{F_x}}\rightarrow \mathrm{Bundle}_{G}\times \mathrm{Stack}_{\mathrm{classify},\Pi_I W^\mathrm{Tannakian}_{F_x}}.
\end{align}
For any representation $R$ of the Langlands group $G^\mathrm{Lan}$ in the coefficient $A$ we define the corresponding Hecke operator $\mathrm{Hecke}(\square)$ in local setting by using the categoricalized Fourier-transformation fashion functor by first taking the corresponding pushforward along $f_A$ of any complex then take the condensed tensor product with the sheaf corresponding to $R$ over the Hecke stack, then take the corresponding push-forward along:
\begin{align}
f_C: D(\mathrm{Stack}_{\mathrm{Hecke},G,I})_{\mathrm{arithmetic-D}} \rightarrow D(\mathrm{Bundle}_{G}\times \Pi_I\mathrm{Stack}_{\mathrm{Cartier},W^\mathrm{Tannakian}_{F_x}})_{\mathrm{arithmetic-D}} \\
 \rightarrow D(\mathrm{Bundle}_{G}\times \mathrm{Stack}_{\mathrm{classify},\Pi_I W^\mathrm{Tannakian}_{F_x}})_{\mathrm{arithmetic-D}}.
\end{align}
Namely we use the notation $\mathrm{Hecke}(\square)$ for:
\begin{align}
\mathrm{pushforward}_{f_B}(\mathrm{pullback}_{f_A}\square\otimes_{\blacksquare,\mathrm{complete}}\mathcal{F}_R).
\end{align}
In this local setting the corresponding operators are well-defined.
\end{theorem}

\begin{conjecture}
Furthermore assume conjecture on the 6-functor formalim for solid quasi-coherent sheaves over $v$-topology. For any representation $R$ of the Langlands group $G^\mathrm{Lan}$ in the coefficient $A$ we conjecture we can define the corresponding Hecke operator $\mathrm{Hecke}(\square)$ in mixed-parity setting by using the categoricalized Fourier-transformation fashion functor by first taking the corresponding pushforward along $f_A$ of any complex then take the condensed tensor product with the sheaf corresponding to $R$ over the Hecke stack, then take the corresponding push-forward along:
\begin{align}
f_B: D(\mathrm{Stack}_{\mathrm{Hecke},G,2,I})_{\text{KL},\mathrm{lisse},\blacksquare,\widetilde{C}_\blacksquare} \rightarrow D(\mathrm{Bundle}_{G,2}\times \Pi_I \mathrm{Stack}_{\mathrm{Cartier},W_{K,2}})_{\text{KL},\mathrm{lisse},\blacksquare,\widetilde{C}_\blacksquare}.
\end{align}
Namely we use the notation $\mathrm{Hecke}(\square)$ for:
\begin{align}
\mathrm{pushforward}_{f_B}(\mathrm{pullback}_{f_A}\square\otimes_{\blacksquare,\mathrm{complete}}\mathcal{F}_R).
\end{align}
The corresponding Hecke operation in this setting is well-defined. The corresponding image lies in the corresponding $W_{K,2}$-invariant sub-derived $\infty$-category:
\begin{align}
D(\mathrm{Bundle}_{G,2})_{\text{KL},\mathrm{lisse},\blacksquare,\widetilde{C}_\blacksquare}.
\end{align}
We conjecture the Satake formalism can be proved in this situation in order to make this definition well-defined.
\end{conjecture}

\begin{theorem}
For any finite extension $K$ over $\mathbb{Q}_p$ or $\mathbb{F}_q((z))$ with $\ell\neq p$, $p>2$ we have well-defined Schur-irreducible sheaves in the corresponding derived $\infty$-category $D(\mathrm{Bundle}_{G,2})_{\text{lisse},\blacksquare,A}$ which are containing all the corresponding Schur-irreducible sheaves in the corresponding derived $\infty$-category $D(\mathrm{Bundle}_{G})_{\text{lisse},\blacksquare,A}$ in the usual setting in \cite{FS} in certain functorial manner, while the latters are related directly to the corresponding smooth representations of the corresponding reductive group $G(K)$ by considering the Bernstein center. Then for any such Schur-irreducible complex there exist some $W_{K,2}$-parameterization into Langlands dual group in well-defined functorial way, compatible with the usual Fargues-Scholze correspondence. Here we make the same requirement on $A$ as in \cite{FS} which is chosen to be an algebraic closure of the $\ell$-adic field containing the square root of the residual field cardinality.
\end{theorem}

\begin{theorem}
For any algebraic curve $X$ over finite field, and we fix a point $x\in X$. we have well-defined Schur-irreducible sheaves in the corresponding derived $\infty$-category $D(\mathrm{Bundle}_{G})_{\mathrm{arithmetic-D}}$. Then for any such Schur-irreducible complex there exist some $W^\mathrm{Tannakian}_{F_x}$-parameterization into Langlands dual group in well-defined functorial way, compatible with the usual Abe's correspondence for $\mathrm{GL}$. Here the corresponding Tannakian groups are the Tannakian groups for the overconvergent isocrystals in the local setting as in \cite{AI}, \cite{KI}, \cite{KXII}, \cite{DK}.
\end{theorem}

We extend the Langlands program in various subprograms with certain different generalizations: 
\begin{setting}
(1) Mixed-parity functorial perturbation of the usual Langlands program after Fargues-Scholze in all characteristics. In this setting the corresponding $\mathbb{Z}_\ell$-linear derived $\infty$-category is the corresponding derived $\infty$-category:
\begin{align}
D(\mathrm{Bundle}_{G,2})_{\text{lisse},\blacksquare,A}.
\end{align}

\end{setting}

\begin{setting}
(2) Robba-Frobenius sheafified functorial perturbation of the usual Langlands program after Fargues-Scholze and Kedlaya-Liu in all characteristics. In this setting the corresponding $\mathbb{Z}_\ell$-linear derived $\infty$-category is the corresponding derived $\infty$-category:
\begin{align}
D(\mathrm{Bundle}_{G,2})_{\text{KL},\mathrm{lisse},\blacksquare,\widetilde{C}_\blacksquare}.
\end{align}

\end{setting}

\begin{setting}
(3) Global arithmetic $D$-module theoretic functorial perturbation of the usual Langlands program after Fargues-Scholze and Abe-Kedlaya-Xu. We work over a curve over finite field in charateristic $p>0$. In this setting the corresponding $\mathbb{Z}_p$-linear derived $\infty$-category is the corresponding derived $\infty$-category:
\begin{align}
D(\mathrm{Bundle}^\mathrm{global}_{G})_{\mathrm{arithmetic-D}}.
\end{align}
 
\end{setting}

\begin{setting}
(4) Local arithemtic $D$-module theoretic functorial perturbation of the usual Langlands program after Fargues-Scholze and Abe-Kedlaya-Xu. We work over a curve over finite field in charateristic $p>0$. In this setting the corresponding $\mathbb{Z}_p$-linear derived $\infty$-category is the corresponding derived $\infty$-category:
\begin{align}
D(\mathrm{Bundle}_{G})_{\mathrm{arithmetic-D}}.
\end{align}
Here we still consider the local analytic stacks as in (1) and (2), but we treat them as the corresponding formal stacks since we are in the equal characteristic situation. 
\end{setting}

\begin{remark}
These two derived $\infty$-categories of arithmetic $D$-modules (in the setting (3), (4) above) are those $\infty$-categorical versions of the categories in \cite{AI}, where one has to take the inductive categories first then consider those complexes with cohomology in the original abelian category of the $D$-modules. Note this is already some problem which is considered in \cite{AI}.
\end{remark}

\begin{remark}
In the corresponding formalization after \cite{VLa} and \cite{FS} one should consider the corresponding picture in a uniformed manner where we do have certain derived $\infty$-category by the corresponding Grothendieck categorical consideration. Then our idea is to construct certain geometric space where the corresponding Bernstein center can live. For instance in \cite{VLa} the corresponding construction is in $\ell$-adic perverse setting where eventually the corresponding Bernstein center will realize the corresponding correspondence. To be more precise any smooth representation will be mapped directly into the corresponding Bernstein center which implies cleverly the main conjecture from the smooth side to the corresponding Weil representation side.
\end{remark}

We have the formalism from \cite{VLa} (see \cite[Chapter VIII Theorem 4.1 and Chapter IX Proposition 4.1]{FS}) for any $\mathbb{Z}_\ell$-linear category namely we can realize a morphism from the corresponding cocycles for any discrete group to the corresponding Bernstein center for any such category $D$:
\begin{align}
\mathrm{BerCenter}_D
\end{align}
after taking the corresponding colimit thorough our all finite sub groups.

\begin{remark}
We start from the corresponding Hecke operations in the $\infty$-categoricalized Fourier-transformation manner. All the four settings can be formalized by consider the corresponding Weil groups:
\begin{align}
W_{K,2},W_{F_X},W_{F_x}.
\end{align}
Here $F_X$ and $F_x$ are the function fields for the curves. The corresponding Weil group equivariance in our setting is realized by using the corresponding classifier stack:
\begin{align}
\mathrm{Stack}_{\mathrm{classify},*}, 
\end{align}
where $*$ is one of:
\begin{align}
W_{K,2},W_{F_X},W_{F_x}.
\end{align}
There are certain morphisms from the stack of all the corresponding rank 1 Cartier divisors to the corresponding classifying stacks as in the above. We use the corresponding notation: 
\begin{align}
\mathrm{Stack}_{\mathrm{Cartier},*}, 
\end{align}
where $*$ is one of:
\begin{align}
W_{K,2},W_{F_X},W_{F_x}
\end{align}
for the Cartier divisor stack as in \cite{FS}.
\end{remark}

\begin{definition}
For the corresponding setting (1) and (2) we have the corresponding Hecke stacks $\mathrm{Stack}_{\mathrm{Hecke},G,2,I}$ for any finite set $I$. This Hecke stack can be defined by taking the corresponding pull-back along the corresponding morphism from $\mathrm{Bundle}_{G,2}$ to the corresponding stack $\mathrm{Bundle}_{G}$. We have then the corresponding morphisms in the following:
\begin{align}
f_A: \mathrm{Stack}_{\mathrm{Hecke},G,2,I} \rightarrow \mathrm{Bundle}_{G,2}
\end{align}
with
\begin{align}
f_B: \mathrm{Stack}_{\mathrm{Hecke},G,2,I} \rightarrow \mathrm{Bundle}_{G,2}\times \Pi_I \mathrm{Stack}_{\mathrm{Cartier},W_{K,2}}
\end{align}
with
\begin{align}
f_C: \mathrm{Stack}_{\mathrm{Hecke},G,2,I} \rightarrow \mathrm{Bundle}_{G,2}\times \Pi_I \mathrm{Stack}_{\mathrm{Cartier},W_{K,2}}\rightarrow \mathrm{Bundle}_{G,2}\times \mathrm{Stack}_{\mathrm{classify},\Pi_I W_{K,2}}.
\end{align}
For any representation $R$ of the Langlands group $G^\mathrm{Lan}$ in the coefficient $A$ we define the corresponding Hecke operator $\mathrm{Hecke}(\square)$ in mixed-parity setting by using the categoricalized Fourier-transformation fashion functor by first taking the corresponding pushforward along $f_A$ of any complex then take the condensed tensor product with the sheaf corresponding to $R$ over the Hecke stack, then take the corresponding push-forward along:
\begin{align}
f_B: D(\mathrm{Stack}_{\mathrm{Hecke},G,2,I})_{\text{lisse},\blacksquare,A} \rightarrow D(\mathrm{Bundle}_{G,2}\times \Pi_I \mathrm{Stack}_{\mathrm{Cartier},W_{K,2}})_{\text{lisse},\blacksquare,A}.
\end{align}
Namely we use the notation $\mathrm{Hecke}(\square)$ for:
\begin{align}
\mathrm{pushforward}_{f_B}(\mathrm{pullback}_{f_A}\square\otimes_{\blacksquare,\mathrm{complete}}\mathcal{F}_R).
\end{align}
For any representation $R$ of the Langlands group $G^\mathrm{Lan}$ in the coefficient $A$ we define the corresponding Hecke operator $\mathrm{Hecke}(\square)$ in mixed-parity setting by using the categoricalized Fourier-transformation fashion functor by first taking the corresponding pushforward along $f_A$ of any complex then take the condensed tensor product with the sheaf corresponding to $R$ over the Hecke stack, then take the corresponding push-forward along:
\begin{align}
f_B: D(\mathrm{Stack}_{\mathrm{Hecke},G,2,I})_{\text{KL},\mathrm{lisse},\blacksquare,\widetilde{C}_\blacksquare} \rightarrow D(\mathrm{Bundle}_{G,2}\times \Pi_I \mathrm{Stack}_{\mathrm{Cartier},W_{K,2}})_{\text{KL},\mathrm{lisse},\blacksquare,\widetilde{C}_\blacksquare}.
\end{align}
Namely we use the notation $\mathrm{Hecke}(\square)$ for:
\begin{align}
\mathrm{pushforward}_{f_B}(\mathrm{pullback}_{f_A}\square\otimes_{\blacksquare,\mathrm{complete}}\mathcal{F}_R).
\end{align}
Here we assume such $\mathcal{F}_R$ can be defined in the category of all the Frobenius sheaves here over the Robba sheaves. And at the moment we assume the six-functor formalism holds for the solid lisse complexes over solidified Robba sheaves.
\end{definition}

Then in the geometric setting we have the following:

\begin{definition}
For the corresponding setting (4) we have the corresponding Hecke stacks $\mathrm{Stack}_{\mathrm{Hecke},G,I}$ for any finite set $I$.We have then the corresponding morphisms in the following:
\begin{align}
f_A: \mathrm{Stack}_{\mathrm{Hecke},G,I} \rightarrow \mathrm{Bundle}_{G}
\end{align}
with
\begin{align}
f_B: \mathrm{Stack}_{\mathrm{Hecke},G,I} \rightarrow \mathrm{Bundle}_{G}\times \Pi_I \mathrm{Stack}_{\mathrm{Cartier},W^\mathrm{Tannakian}_{F_x}}
\end{align}
with
\begin{align}
f_C: \mathrm{Stack}_{\mathrm{Hecke},G,I} \rightarrow \mathrm{Bundle}_{G}\times \Pi_I \mathrm{Stack}_{\mathrm{Cartier},W^\mathrm{Tannakian}_{F_x}}\rightarrow \mathrm{Bundle}_{G}\times \mathrm{Stack}_{\mathrm{classify},\Pi_I W^\mathrm{Tannakian}_{F_x}}.
\end{align}
Here all the stacks are those formal stackification of the stacks in \cite{FS}, since in the local setting (4), we are in the equal characteristic situation, and we do have the 6-functor formalism from \cite{AI}. The stack:
\begin{align}
\mathrm{Stack}_{\mathrm{Cartier},W^\mathrm{Tannakian}_{F_x}}
\end{align}
is the moduli Tannakian stack of arithmetic $D$-modules (which are actually just $F$-isocrystals) over the formal Cartier stack (which is the formal stackification of the corresponding stack in \cite{FS}). 
For any representation $R$ of the Langlands group $G^\mathrm{Lan}$ in the coefficient $A$ we define the corresponding Hecke operator $\mathrm{Hecke}(\square)$ in local setting by using the categoricalized Fourier-transformation fashion functor by first taking the corresponding pushforward along $f_A$ of any complex then take the condensed tensor product with the sheaf corresponding to $R$ over the Hecke stack, then take the corresponding push-forward along:
\begin{align}
f_B: D(\mathrm{Stack}_{\mathrm{Hecke},G,I})_{\mathrm{arithmetic-D}} \rightarrow D(\mathrm{Bundle}_{G}\times \Pi_I \mathrm{Stack}_{\mathrm{Cartier},W^\mathrm{Tannakian}_{F_x}})_{\mathrm{arithmetic-D}}.
\end{align}
Namely we use the notation $\mathrm{Hecke}(\square)$ for:
\begin{align}
\mathrm{pushforward}_{f_B}(\mathrm{pullback}_{f_A}\square\otimes_{\blacksquare,\mathrm{complete}}\mathcal{F}_R).
\end{align}

\end{definition}

\begin{definition}
In the settings (1), (2) and (4), we can define the corresponding $\infty$-categorical Hecke eigensheaves as in the usual geometric Langlands as well. Those are the complexes which are retained as external tensor product with $\mathcal{F}_R$.
\end{definition}

\indent Then we can formulate the following theorems:
\begin{theorem}\label{theorem17}
For the corresponding setting (1) and (2) we have the corresponding Hecke stacks $\mathrm{Stack}_{\mathrm{Hecke},G,2,I}$ for any finite set $I$. This Hecke stack can be defined by taking the corresponding pull-back along the corresponding morphism from $\mathrm{Bundle}_{G,2}$ to the corresponding stack $\mathrm{Bundle}_{G}$. We have then the corresponding morphisms in the following:
\begin{align}
f_A: \mathrm{Stack}_{\mathrm{Hecke},G,2,I} \rightarrow \mathrm{Bundle}_{G,2}
\end{align}
with
\begin{align}
f_B: \mathrm{Stack}_{\mathrm{Hecke},G,2,I} \rightarrow \mathrm{Bundle}_{G,2}\times \Pi_I \mathrm{Stack}_{\mathrm{Cartier},W_{K,2}}
\end{align}
with
\begin{align}
f_C: \mathrm{Stack}_{\mathrm{Hecke},G,2,I} \rightarrow \mathrm{Bundle}_{G,2}\times\Pi_I  \mathrm{Stack}_{\mathrm{Cartier},W_{K,2}}\rightarrow \mathrm{Bundle}_{G,2}\times \mathrm{Stack}_{\mathrm{classify},\Pi_I W_{K,2}}.
\end{align}
The corresponding Hecke operation in this setting is well-defined. The corresponding image lies in the corresponding $W_{K,2}$-invariant sub-derived $\infty$-category:
\begin{align}
D(\mathrm{Bundle}_{G,2})_{\text{lisse},\blacksquare,A}.
\end{align}
For any representation $R$ of the Langlands group $G^\mathrm{Lan}$ in the coefficient $A$ we define the corresponding Hecke operator $\mathrm{Hecke}(\square)$ in mixed-parity setting by using the categoricalized Fourier-transformation fashion functor by first taking the corresponding pushforward along $f_A$ of any complex then take the condensed tensor product with the sheaf corresponding to $R$ over the Hecke stack, then take the corresponding push-forward along:
\begin{align}
f_B: D(\mathrm{Stack}_{\mathrm{Hecke},G,2,I})_{\text{lisse},\blacksquare,A} \rightarrow D(\mathrm{Bundle}_{G,2}\times \Pi_I \mathrm{Stack}_{\mathrm{Cartier},W_{K,2}})_{\text{lisse},\blacksquare,A}.
\end{align}
Namely we use the notation $\mathrm{Hecke}(\square)$ for:
\begin{align}
\mathrm{pushforward}_{f_B}(\mathrm{pullback}_{f_A}\square\otimes_{\blacksquare,\mathrm{complete}}\mathcal{F}_R).
\end{align}
\end{theorem}

\begin{proof}
Here the analog of the Satake formalism in our setting from \cite{FS} actually goes in the sense that we can pull any sheaf $\mathcal{F}_R$ associated with a representation $R$ of the Langlands dual group along the morphisms:
\begin{align}
\mathrm{Stack}_{\mathrm{Hecke},G,2,I} \rightarrow \mathrm{Stack}_{\mathrm{Hecke},G,I}.
\end{align} 
The corresponding well-definedness are directly from the corresponding machinery in \cite{FS} constructed for all generalized $v$-stacks. As in \cite[Corollary 2.3 in Chapter IX]{FS} (after we assume the conjecture on the 6-functors) one can use induction to further reduce to the situation where $I$ is singleton. Then in this case one compares the two categories:
\begin{align}
D(\mathrm{Bundle}_{G,2})_{\text{lisse},\blacksquare,A},D(\mathrm{Bundle}_{G,2}\times X_\mathcal{C})_{\text{lisse},\blacksquare,A}
\end{align}
where $C$ is the $z$-adic or $p$-adic complex number field in the first situation, and then one compares the two categories:
\begin{align}
D(\mathrm{Bundle}_{G,2})_{\text{KL},\mathrm{lisse},\blacksquare,\widetilde{C}_\blacksquare},D(\mathrm{Bundle}_{G,2}\times X_\mathcal{C})_{\text{KL},\mathrm{lisse},\blacksquare,\widetilde{C}_\blacksquare}
\end{align}
where $C$ is the $z$-adic or $p$-adic complex number field in the second  situation. In the $\ell$-adic situation, one can also prove this by using the corresponding functoriality from \cite{FS} by using the results in the usual situation. Namely we look at the corresponding Hecke operator in the mixed-parity situation:
\begin{align}
D(\mathrm{Bundle}_{G,2})_{\text{lisse},\blacksquare,A} \rightarrow D(\mathrm{Bundle}_{G,2}\times \Pi_I \mathrm{Stack}_{\mathrm{Cartier},W_{K,2}})_{\text{lisse},\blacksquare,A}
\end{align}
with the usual Hecke operator in \cite{FS}:
\begin{align}
D(\mathrm{Bundle}_{G})_{\text{lisse},\blacksquare,A} \rightarrow D(\mathrm{Bundle}_{G}\times \Pi_I \mathrm{Stack}_{\mathrm{Cartier},W_{K}})_{\text{lisse},\blacksquare,A}
\end{align}
where they form a commutative diagram by using the functoriality from the \cite{FS}. Then we can get the well-definedness of this map with the desired image lying in the category of $W_{K,2}$-equivariant objects. In the Robba sheaves situation, we have no such issue since we assume the corresponding six-operator with desired Endlichkeitssatz for lisse Robba sheaf complexes in the desired category.
\end{proof}

\begin{conjecture}
Furthermore assume conjecture on the 6-functor formalim for solid quasi-coherent sheaves over $v$-topology. For any representation $R$ of the Langlands group $G^\mathrm{Lan}$ in the coefficient $A$ we conjecture we can define the corresponding Hecke operator $\mathrm{Hecke}(\square)$ in mixed-parity setting by using the categoricalized Fourier-transformation fashion functor by first taking the corresponding pushforward along $f_A$ of any complex then take the condensed tensor product with the sheaf corresponding to $R$ over the Hecke stack, then take the corresponding push-forward along:
\begin{align}
f_B: D(\mathrm{Stack}_{\mathrm{Hecke},G,2,I})_{\text{KL},\mathrm{lisse},\blacksquare,\widetilde{C}_\blacksquare} \rightarrow D(\mathrm{Bundle}_{G,2}\times \Pi_I \mathrm{Stack}_{\mathrm{Cartier},W_{K,2}})_{\text{KL},\mathrm{lisse},\blacksquare,\widetilde{C}_\blacksquare}.
\end{align}
Namely we use the notation $\mathrm{Hecke}(\square)$ for:
\begin{align}
\mathrm{pushforward}_{f_B}(\mathrm{pullback}_{f_A}\square\otimes_{\blacksquare,\mathrm{complete}}\mathcal{F}_R).
\end{align}
The corresponding Hecke operation in this setting is well-defined. The corresponding image lies in the corresponding $W_{K,2}$-invariant sub-derived $\infty$-category:
\begin{align}
D(\mathrm{Bundle}_{G,2})_{\text{KL},\mathrm{lisse},\blacksquare,\widetilde{C}_\blacksquare}.
\end{align}
We conjecture the Satake formalism can be proved in this situation in order to make this definition well-defined.
\end{conjecture}

\begin{theorem}
For the corresponding setting (4) we have the corresponding Hecke stacks $\mathrm{Stack}_{\mathrm{Hecke},G,I}$ for any finite set $I$.We have then the corresponding morphisms in the following:
\begin{align}
f_A: \mathrm{Stack}_{\mathrm{Hecke},G,I} \rightarrow \mathrm{Bundle}_{G}
\end{align}
with
\begin{align}
f_B: \mathrm{Stack}_{\mathrm{Hecke},G,I} \rightarrow \mathrm{Bundle}_{G}\times \Pi_I \mathrm{Stack}_{\mathrm{Cartier},W^\mathrm{Tannakian}_{F_x}}
\end{align}
with
\begin{align}
f_C: \mathrm{Stack}_{\mathrm{Hecke},G,I} \rightarrow \mathrm{Bundle}_{G}\times \Pi_I \mathrm{Stack}_{\mathrm{Cartier},W^\mathrm{Tannakian}_{F_x}}\rightarrow \mathrm{Bundle}_{G}\times \mathrm{Stack}_{\mathrm{classify},\Pi_I W^\mathrm{Tannakian}_{F_x}}.
\end{align}
For any representation $R$ of the Langlands group $G^\mathrm{Lan}$ in the coefficient $A$ we define the corresponding Hecke operator $\mathrm{Hecke}(\square)$ in local setting by using the categoricalized Fourier-transformation fashion functor by first taking the corresponding pushforward along $f_A$ of any complex then take the condensed tensor product with the sheaf corresponding to $R$ over the Hecke stack, then take the corresponding push-forward along:
\begin{align}
f_C: D(\mathrm{Stack}_{\mathrm{Hecke},G,I})_{\mathrm{arithmetic-D}} \rightarrow D(\mathrm{Bundle}_{G}\times \Pi_I\mathrm{Stack}_{\mathrm{Cartier},W^\mathrm{Tannakian}_{F_x}})_{\mathrm{arithmetic-D}} \\
 \rightarrow D(\mathrm{Bundle}_{G}\times \mathrm{Stack}_{\mathrm{classify},\Pi_I W^\mathrm{Tannakian}_{F_x}})_{\mathrm{arithmetic-D}}.
\end{align}
Namely we use the notation $\mathrm{Hecke}(\square)$ for:
\begin{align}
\mathrm{pushforward}_{f_B}(\mathrm{pullback}_{f_A}\square\otimes_{\blacksquare,\mathrm{complete}}\mathcal{F}_R).
\end{align}
In both setting the corresponding operations are well-defined. The image of the functor lies in the $W^\mathrm{Tannakian}$-equivarient arithmetic $D$-modules over the $\mathrm{Bundle}_G$.
\end{theorem}

\begin{proof}
The corresponding well-definedness are directly from the corresponding machinery in \cite{AI} constructed for all generalized $v$-stacks. By \cite[Section 3]{XZ} the Satake formalism directly sends any representation of the Langlands dual group to a sheaf over the Hecke stacks. We are in the local setting and the local stacks we are considering are actually formal discs with Frobenius quotients. Then following \cite{VLa}, \cite{FS} one argues as in \cref{theorem17} to reduce first to the situation where $I$ is a singleton, then passing to the algebraic closure $\mathbb{C}$ in the current situation to derive that the image of the functor will be the $W^\mathrm{Tannakian}$-equivariant arithmetic $D$-modules.
 \end{proof}

\begin{remark}
The corresponding formalism in \cite[Chapter VIII Theorem 4.1, Chapter IX Proposition 4.1, and the proof]{FS} is actually can be applied in the scenario when we have any discrete group and any $\mathbb{Z}_\ell$-linear categories. And we \textit{conjecture} by using the corresponding formalism we can actually give a proof on the corresponding $\mathbb{Z}_p$-linear categorical situation for arithmetic $D$-modules.
\end{remark}

We have the formalism from \cite{VLa} (see \cite[Chapter VIII Theorem 4.1 and Chapter IX Proposition 4.1]{FS}) for any $\mathbb{Z}_p$-linear category namely we can realize a morphism from the corresponding cocycles for any discrete group to the corresponding Bernstein center for any such category $D$:
\begin{align}
\mathrm{BerCenter}_D
\end{align}
after taking the corresponding colimit thorough our all finite sub groups. See \cite[Chapter VIII Theorem 4.1 and Chapter IX Proposition 4.1]{FS}, literally after \cite{VLa}.

\begin{theorem}
For any finite extension $K$ over $\mathbb{Q}_p$ or $\mathbb{F}_q((z))$ with $\ell\neq p$, $p>2$ we have well-defined Schur-irreducible sheaves in the corresponding derived $\infty$-category $D(\mathrm{Bundle}_{G,2})_{\text{lisse},\blacksquare,A}$ which are containing all the corresponding Schur-irreducible sheaves in the corresponding derived $\infty$-category $D(\mathrm{Bundle}_{G})_{\text{lisse},\blacksquare,A}$ in the usual setting in \cite{FS} in certain functorial manner, while the latters are related directly to the corresponding smooth representations of the corresponding reductive group $G(K)$ by considering the Bernstein center. Then for any such Schur-irreducible complex there exist some $W_{K,2}$-parameterization into Langlands dual group in well-defined functorial way, compatible with the usual Fargues-Scholze correspondence. Here we make the same requirement on $A$ as in \cite{FS} which is chosen to be an algebraic closure of the $\ell$-adic field containing the square root of the residual field cardinality.
\end{theorem}

\begin{proof}
Indeed, $W_{K,2}$ as in \cite{FS} can be discretized under open subgroups of two-fold covering of the wild inertia in our situation, where the discretization is following \cite{FS}, which sends us into the corresponding formalism from Lafforgue and Fargues-Scholze. In the discrete setting the formalism is basically just direct application and in the general situation the proof can actually given in the completely same fashion. See \cite[Chapter VIII Theorem 4.1 and Chapter IX Proposition 4.1]{FS}, literally after \cite{VLa}.
\end{proof}

\begin{theorem}
For any algebraic curve $X$ over finite field, and we fix a point $x\in X$. we have well-defined Schur-irreducible sheaves in the corresponding derived $\infty$-category $D(\mathrm{Bundle}_{G})_{\mathrm{arithmetic-D}}$. Then for any such Schur-irreducible complex there exist some $W^\mathrm{Tannakian}_{F_x}$-parameterization into Langlands dual group in well-defined functorial way, compatible with the usual Abe's correspondence for $\mathrm{GL}$. Here the corresponding Tannakian groups are the Tannakian groups for the overconvergent isocrystals in the local setting as in \cite{AI}, \cite{KI}, \cite{KXII}, \cite{DK}.
\end{theorem}

\begin{proof}
In fact $W^\mathrm{Tannakian}_{F_x}$ is actually algebraic group, which sends us into the corresponding formalism from Lafforgue and Fargues-Scholze. The proof can actually given in the completely same fashion. See \cite[Chapter VIII Theorem 4.1 and Chapter IX Proposition 4.1]{FS}, literally after \cite{VLa}.
\end{proof}

\begin{remark}
Here the proof of this theorem is essentially relying on the formalism of \cite{FS}, where we consider the $\mathrm{Bundle}_G$ as a formal stack from \cite{FS} while keeping all the stacks involved available while regarding theorem as formal stacks (not $v$-stacks). This is reasonable since we are in the corresponding equal characteristic situation. However we believe there is another proof by using the formalism in \cite{GL}, i.e. the corresponding modulis of shtukas in the local setting. The basic idea here follows again essentially \cite{FS}: in the $\ell$-adic setting, since \cite{FS} used the perfectoid framework simplifies the proof of \cite{GL} into a straighforward local proof by using the $v$-stacks involved, one should regard as the strategy we adpoted as an arithmetic $D$-module picture parallel to and after this.
\end{remark}

\begin{remark}
Then we can consider the corresponding generalization of this picture further after \cite{FS}, such as the corresponding spectral action conjecture by considering certain moduli stack on the other side, again after \cite{AI}, \cite{KI} and \cite{KXII}. After \cite{DHKM}, \cite{FS}, \cite{Z}, one can see that the corresponding construction can be directly generalized to our setting by discretization of the two-fold covering Weil groups in the perfectoid situation, namely by using those discrete quotient to form substacks of the full stacks.
\end{remark}

\newpage
\chapter{Motivic Approach to Generalized Langlands Program}

\noindent In this section we prove the following theorems which are directly generalization of Scholze's motivic approach to local Langlands in \cite{Scho1} after \cite{V1}, \cite{A1}, \cite{RS}, \cite{Scho2}:

\begin{theorem}
For the current setting we have the corresponding Hecke stacks $\mathrm{Stack}_{\mathrm{Hecke},G,2,I}$ for any finite set $I$. This Hecke stack can be defined by taking the corresponding pull-back along the corresponding morphism from $\mathrm{Bundle}_{G,2}$ to the corresponding stack $\mathrm{Bundle}_{G}$. We have then the corresponding morphisms in the following:
\begin{align}
f_A: \mathrm{Stack}_{\mathrm{Hecke},G,2,I} \rightarrow \mathrm{Bundle}_{G,2}
\end{align}
with
\begin{align}
f_B: \mathrm{Stack}_{\mathrm{Hecke},G,2,I} \rightarrow \mathrm{Bundle}_{G,2}\times \Pi_I \mathrm{Stack}_{\mathrm{Cartier},M_{K,2}}
\end{align}
with
\begin{align}
f_C: \mathrm{Stack}_{\mathrm{Hecke},G,2,I} \rightarrow \mathrm{Bundle}_{G,2}\times \Pi_I \mathrm{Stack}_{\mathrm{Cartier},M_{K,2}}\rightarrow \mathrm{Bundle}_{G,2}\times \mathrm{Stack}_{\mathrm{classify},\Pi_I M_{K,2}}.
\end{align}
The corresponding Hecke operation in this setting is well-defined. The corresponding image lies in the corresponding $M_{K,2}$-invariant sub-derived $\infty$-category:
\begin{align}
D(\mathrm{Bundle}_{G,2})_{\text{motivic},\blacksquare}.
\end{align}
For any representation $R$ of the Langlands group $G^\mathrm{Lan}$ we define the corresponding Hecke operator $\mathrm{Hecke}(\square)$ in mixed-parity setting by using the categoricalized Fourier-transformation fashion functor by first taking the corresponding pushforward along $f_A$ of any complex then take the condensed tensor product with the sheaf corresponding to $R$ over the Hecke stack, then take the corresponding push-forward along:
\begin{align}
f_B: D(\mathrm{Stack}_{\mathrm{Hecke},G,2,I})_{\text{motivic},\blacksquare} \rightarrow D(\mathrm{Bundle}_{G,2}\times \Pi_I \mathrm{Stack}_{\mathrm{Cartier},M_{K,2}})_{\text{motivic},\blacksquare}.
\end{align}
Namely we use the notation $\mathrm{Hecke}(\square)$ for:
\begin{align}
\mathrm{pushforward}_{f_B}(\mathrm{pullback}_{f_A}\square\otimes_{\blacksquare,\mathrm{complete}}\mathcal{F}_R).
\end{align}
Here $\mathrm{Stack}_{\mathrm{Cartier},M_{K,2}}$ is the corresponding generalized motivic Cartier stack after \cite{Scho1} with the corresponding motivic Weil group $M_{K,2}$. This is defined by using tensor monoidal $\infty$-category of motivic sheaves over the $\mathrm{Stack}_{\mathrm{Cartier},W_{K,2}}$, i.e. the Tannakian group in this situation. For more detail see \cite{A2}, \cite{A3}, \cite{A4}. One takes the localization of the generalized motivic category of $\infty$-sheaves over $v$-site of Banach rings to the Cartier stack in our setting, which will produces the corresponding motivic generalized Cartier stack in our setting.
\end{theorem}

\begin{remark}
The motivic Galois group in \cite{Scho1} admits a quotient map to the motivic Galois group of $K$, which allows one to define the corresponding two fold covering of the motivic Galois group. In our notation it will be the corresponding fibre product of $W_{K,2}$ and $M_K$, for instance over $G_{K}$. 
\end{remark}

\begin{theorem} \mbox{\textbf{(After Scholze, \cite{Scho1}, \cite{Scho2})}}
For any finite extension $K$ over $\mathbb{Q}_p$ or $\mathbb{F}_q((z))$ with $\ell\neq p$, $p>2$ we have well-defined Schur-irreducible sheaves (i.e. specialize to a Schur-irreducible object for each $\ell$) in the corresponding derived $\infty$-category $D(\mathrm{Bundle}_{G,2})_{\text{motivic},\blacksquare}$ which are containing all the corresponding Schur-irreducible sheaves in the corresponding derived $\infty$-category $D(\mathrm{Bundle}_{G})_{\text{motivic},\blacksquare}$ in the usual setting in \cite{Scho1} in certain functorial manner, while the latters are related directly to the corresponding smooth representations of the corresponding reductive group $G(K)$ by considering the Bernstein center, i.e. the motivic $G(K)$-equivariant sheaves.  
\end{theorem}

\begin{conjecture} \mbox{\textbf{(After Scholze, \cite{Scho1}, \cite{Scho2})}}
 For any finite extension $K$ over $\mathbb{Q}_p$ or $\mathbb{F}_q((z))$ with $\ell\neq p$, $p>2$ we have well-defined Schur-irreducible sheaves in the corresponding derived $\infty$-category $D(\mathrm{Bundle}_{G,2})_{\text{motivic},\blacksquare}$ which are containing all the corresponding Schur-irreducible sheaves in the corresponding derived $\infty$-category $D(\mathrm{Bundle}_{G})_{\text{motivic},\blacksquare}$ in the usual setting in \cite{Scho1} in certain functorial manner, while the latters are related directly to the corresponding smooth representations of the corresponding reductive group $G(K)$ by considering the Bernstein center, i.e. the motivic $G(K)$-equivariant sheaves which can be specialized to $G(K)$-representations over $\overline{\mathbb{Q}}_\ell$ for each $\ell$. We then conjecture we can generalize \cite{Scho1} in some well-defined sense to our generalized setting, by considering the $M_{K,2}$-group. 
 \end{conjecture}

\begin{definition}
For the current setting we have the corresponding Hecke stacks $\mathrm{Stack}_{\mathrm{Hecke},G,2,I}$ for any finite set $I$. This Hecke stack can be defined by taking the corresponding pull-back along the corresponding morphism from $\mathrm{Bundle}_{G,2}$ to the corresponding stack $\mathrm{Bundle}_{G}$. We have then the corresponding morphisms in the following:
\begin{align}
f_A: \mathrm{Stack}_{\mathrm{Hecke},G,2,I} \rightarrow \mathrm{Bundle}_{G,2}
\end{align}
with
\begin{align}
f_B: \mathrm{Stack}_{\mathrm{Hecke},G,2,I} \rightarrow \mathrm{Bundle}_{G,2}\times \Pi_I \mathrm{Stack}_{\mathrm{Cartier},M_{K,2}}
\end{align}
with
\begin{align}
f_C: \mathrm{Stack}_{\mathrm{Hecke},G,2,I} \rightarrow \mathrm{Bundle}_{G,2}\times \Pi_I \mathrm{Stack}_{\mathrm{Cartier},M_{K,2}}\rightarrow \mathrm{Bundle}_{G,2}\times \mathrm{Stack}_{\mathrm{classify},\Pi_I M_{K,2}}.
\end{align}
For any representation $R$ of the Langlands group $G^\mathrm{Lan}$ we define the corresponding Hecke operator $\mathrm{Hecke}(\square)$ in mixed-parity setting by using the categoricalized Fourier-transformation fashion functor by first taking the corresponding pushforward along $f_A$ of any complex then take the condensed tensor product with the sheaf corresponding to $R$ over the Hecke stack, then take the corresponding push-forward along:
\begin{align}
f_B: D(\mathrm{Stack}_{\mathrm{Hecke},G,2,I})_{\text{motivic},\blacksquare} \rightarrow D(\mathrm{Bundle}_{G,2}\times \Pi_I \mathrm{Stack}_{\mathrm{Cartier},M_{K,2}})_{\text{motivic},\blacksquare}.
\end{align}
Namely we use the notation $\mathrm{Hecke}(\square)$ for:
\begin{align}
\mathrm{pushforward}_{f_B}(\mathrm{pullback}_{f_A}\square\otimes_{\blacksquare,\mathrm{complete}}\mathcal{F}_R).
\end{align}
\end{definition}

\begin{definition}
In the current setting, we can define the corresponding $\infty$-categorical Hecke eigensheaves as in the usual geometric Langlands as well. Those are the complexes which are retained as external tensor product with $\mathcal{F}_R$.
\end{definition}

\indent Then we can formulate the following theorems:
\begin{theorem}
For the current setting we have the corresponding Hecke stacks $\mathrm{Stack}_{\mathrm{Hecke},G,2,I}$ for any finite set $I$. This Hecke stack can be defined by taking the corresponding pull-back along the corresponding morphism from $\mathrm{Bundle}_{G,2}$ to the corresponding stack $\mathrm{Bundle}_{G}$. We have then the corresponding morphisms in the following:
\begin{align}
f_A: \mathrm{Stack}_{\mathrm{Hecke},G,2,I} \rightarrow \mathrm{Bundle}_{G,2}
\end{align}
with
\begin{align}
f_B: \mathrm{Stack}_{\mathrm{Hecke},G,2,I} \rightarrow \mathrm{Bundle}_{G,2}\times \Pi_I \mathrm{Stack}_{\mathrm{Cartier},M_{K,2}}
\end{align}
with
\begin{align}
f_C: \mathrm{Stack}_{\mathrm{Hecke},G,2,I} \rightarrow \mathrm{Bundle}_{G,2}\times\Pi_I  \mathrm{Stack}_{\mathrm{Cartier},M_{K,2}}\rightarrow \mathrm{Bundle}_{G,2}\times \mathrm{Stack}_{\mathrm{classify},\Pi_I M_{K,2}}.
\end{align}
The corresponding Hecke operation in this setting is well-defined. The corresponding image lies in the corresponding $M_{K,2}$-invariant sub-derived $\infty$-category:
\begin{align}
D(\mathrm{Bundle}_{G,2})_{\text{motivic},\blacksquare}.
\end{align}
For any representation $R$ of the Langlands group $G^\mathrm{Lan}$ we define the corresponding Hecke operator $\mathrm{Hecke}(\square)$ in mixed-parity setting by using the categoricalized Fourier-transformation fashion functor by first taking the corresponding pushforward along $f_A$ of any complex then take the condensed tensor product with the sheaf corresponding to $R$ over the Hecke stack, then take the corresponding push-forward along:
\begin{align}
f_B: D(\mathrm{Stack}_{\mathrm{Hecke},G,2,I})_{\text{motivic},\blacksquare} \rightarrow D(\mathrm{Bundle}_{G,2}\times \Pi_I \mathrm{Stack}_{\mathrm{Cartier},M_{K,2}})_{\text{motivic},\blacksquare}.
\end{align}
Namely we use the notation $\mathrm{Hecke}(\square)$ for:
\begin{align}
\mathrm{pushforward}_{f_B}(\mathrm{pullback}_{f_A}\square\otimes_{\blacksquare,\mathrm{complete}}\mathcal{F}_R).
\end{align}
\end{theorem}

\begin{proof}
This relies on the existing Satake isomorphisms as in \cite{FS}, \cite{RS}, \cite{Scho1}. Here the analog of the Satake formalism in our setting from \cite{FS} actually goes in the sense that we can pull any sheaf $\mathcal{F}_R$ associated with a representation $R$ of the Langlands dual group along the morphisms:
\begin{align}
\mathrm{Stack}_{\mathrm{Hecke},G,2,I} \rightarrow \mathrm{Stack}_{\mathrm{Hecke},G,I}.
\end{align} 
The corresponding well-definedness are directly from the corresponding machinery in \cite{FS} constructed for all generalized $v$-stacks. As in \cite[Corollary 2.3 in Chapter IX]{FS} (after we assume the conjecture on the 6-functors) one can use induction to further reduce to the situation where $I$ is singleton. Then in this case one compares the two categories:
\begin{align}
D(\mathrm{Bundle}_{G,2})_{\text{motivic},\blacksquare},D(\mathrm{Bundle}_{G,2}\times X_\mathcal{C})_{\text{motivic},\blacksquare}
\end{align}
where $C$ is the $z$-adic or $p$-adic complex number field in the first situation. In the motivic situation, one can also prove this by using the corresponding functoriality from \cite{Scho1} by using the results in the usual situation in \cite{Scho1}. Namely we look at the corresponding Hecke operator in the mixed-parity situation:
\begin{align}
D(\mathrm{Bundle}_{G,2})_{\text{motivic},\blacksquare} \rightarrow D(\mathrm{Bundle}_{G,2}\times \Pi_I \mathrm{Stack}_{\mathrm{Cartier},M_{K,2}})_{\text{motivic},\blacksquare}
\end{align}
with the usual Hecke operator in \cite{FS}:
\begin{align}
D(\mathrm{Bundle}_{G})_{\text{motivic},\blacksquare} \rightarrow D(\mathrm{Bundle}_{G}\times \Pi_I \mathrm{Stack}_{\mathrm{Cartier},W_{K}})_{\text{motivic},\blacksquare}
\end{align}
where they form a commutative diagram by using the functoriality from the \cite{Scho1}. Then we can get the well-definedness of this map with the desired image lying in the category of $M_{K,2}$-equivariant objects. 
\end{proof}

\begin{theorem} \mbox{\textbf{(After Scholze, \cite{Scho1}, \cite{Scho2})}}
For any finite extension $K$ over $\mathbb{Q}_p$ or $\mathbb{F}_q((z))$ with $\ell\neq p$, $p>2$ we have well-defined Schur-irreducible sheaves in the corresponding derived $\infty$-category $D(\mathrm{Bundle}_{G,2})_{\text{motivic},\blacksquare}$ which are containing all the corresponding Schur-irreducible sheaves in the corresponding derived $\infty$-category $D(\mathrm{Bundle}_{G})_{\text{motivic},\blacksquare}$ in the usual setting in \cite{Scho1} in certain functorial manner, while the latters are related directly to the corresponding smooth representations of the corresponding reductive group $G(K)$ by considering the Bernstein center, i.e. the motivic $G(K)$-equivariant sheaves which can be specialized to $G(K)$-representations over $\overline{\mathbb{Q}}_\ell$ for each $\ell$. 
\end{theorem}

\begin{conjecture} \mbox{\textbf{(After Scholze, \cite{Scho1}, \cite{Scho2})}}
For any finite extension $K$ over $\mathbb{Q}_p$ or $\mathbb{F}_q((z))$ with $\ell\neq p$, $p>2$ we have well-defined Schur-irreducible sheaves in the corresponding derived $\infty$-category $D(\mathrm{Bundle}_{G,2})_{\text{motivic},\blacksquare}$ which are containing all the corresponding Schur-irreducible sheaves in the corresponding derived $\infty$-category $D(\mathrm{Bundle}_{G})_{\text{motivic},\blacksquare}$ in the usual setting in \cite{Scho1} in certain functorial manner, while the latters are related directly to the corresponding smooth representations of the corresponding reductive group $G(K)$ by considering the Bernstein center, i.e. the motivic $G(K)$-equivariant sheaves which can be specialized to $G(K)$-representations over $\overline{\mathbb{Q}}_\ell$ for each $\ell$. Then we conjecture we can in some well-defined way generalize \cite{Scho1} to our current setting by using the group $M_{K,2}$. Here $M_{K,2}$ is the fiber product of $M_{K}$ and $W_{K,2}$, see \cite{Scho1}, \cite{A2}, \cite{A3}, \cite{A4}. 
\end{conjecture}

\begin{remark}
Then we can consider the corresponding generalization of this picture further after \cite{FS} and \cite{Scho1}, such as the corresponding spectral action conjecture by considering certain moduli stack on the other side, again after \cite{AI}, \cite{KI} and \cite{KXII}. After \cite{DHKM}, \cite{FS}, \cite{Z}, one can see that the corresponding construction can be directly generalized to our setting by discretization of the two-fold covering Weil groups in the perfectoid situation, namely by using those discrete quotient to form substacks of the full stacks. 
\end{remark}

\newpage
\subsection*{Acknowledgements}
Thanks to both Professor Kedlaya and Professor Sorensen for the corresponding helpful discussion towards the generalized Langlands program, the generalized Langlands dualitization and the generalized Langlands correspondence. The suggestion with respect to and even beyond the key fundamental inspiration from Breuil-Schneider was given to me from Professor Sorensen almost eight years ago in 2016. It is obvious that we gained certain inspiration directly from L.Lafforgue, V.Lafforgue, Scholze and Fargues-Scholze.

\end{document}